\documentclass[11pt]{amsart}

\usepackage{amsmath,amsfonts,amsthm}
\usepackage{amssymb,latexsym}
\usepackage{hyperref}
\newtheorem{theorem}{Theorem}[section]

\newtheorem{lemma}[theorem]{Lemma}
\newtheorem{proposition}[theorem]{Proposition}
\newtheorem{corollary}[theorem]{Corollary}

\theoremstyle{definition}
\newtheorem{definition}[theorem]{Definition}
\theoremstyle{notation}

\theoremstyle{remark}

\numberwithin{equation}{section}

\usepackage{blindtext}

\usepackage{graphicx}
\usepackage{amsmath,amsfonts}
\usepackage{amssymb,latexsym}
\usepackage{hyperref}
\usepackage[all]{xy}
\usepackage{layout}
\usepackage[bottom]{footmisc}
\usepackage{mathrsfs}

\usepackage{tikz-cd}

\begin{document}

%
%
%
%
%

	\title{Topological boundaries of covariant representations}
	
	
	
	\author{Massoud Amini \and Sajad Zavar}
	
	
	%

	\address{
		Faculty of Mathematical Sciences, Tarbiat Modares University, Tehran 14115-134, Iran}
		\email{mamini@modares.ac.ir}           
	\address{
	Faculty of Mathematical Sciences, Kharazmi University, Tehran, Iran}
	\email{s.zavar@modares.ac.ir}           
	%



	\begin{abstract}
	We associate a boundary $\mathcal B_{\pi,u}$ to each covariant representation $(\pi,u,H)$ of $C^*$-dynamical system $(G,A,\alpha)$ and study the action of $G$ on $\mathcal B_{\pi,u}$ and its amenability properties. We relate rigidity properties of traces on the associated crossed product C*-algebra to  faithfulness of action of the group on this boundary.

	\end{abstract}

\maketitle
	
\section{Introduction}

The notion of boundary actions,  introduced by Furstenberg in \cite{f1} and \cite{f2}, is recently used in the study of  rigidity properties of  reduced group C*-algebras generated by the regular representation of discrete groups (more specifically, in problems such as simplicity, uniqueness of trace, and nuclear
embeddings \cite{bkko}, \cite{kk}, \cite{l}). The link between boundary actions and reduced
group C*-algebra was confirmed after the basic observation of Kalantar and Kennedy \cite{kk} that the Furstenberg and Hamana boundaries coincide, namely $C(\partial_FG) = I_G(\mathbb C)$,
where the right hand side is the $G$-injective envelope of complex numbers in the sense of
Hamana \cite{h2}. The properties of $\partial_FG$ are related to certain problems of interest in operator algebras. It is related to the
Ozawa's notion of ``tight nuclear embedding'' \cite{o2} of an
exact C*-algebra in a nuclear C*-algebra, which in turn embeds in the injective envelop of the original exact C*-algebra. Ozawa
proved such a tight embedding for the reduced C*-algebra of the free group $\mathbb F_n$, and Kalantar and Kennedy extended his result to an arbitrary discrete exact group $G$, by showing that the reduced C*-algebra $C_r^*(G)$ of $G$ tightly embeds in the (nuclear) C*-algebra
generated by  $C_r^*(G)$ and $C(\partial_FG)$ in $B(\ell^2(G))$.
The other problems are those of C*-simplicity and unique trace property for the reduced C*-algebra. Again it
is already known by a result of Powers \cite{po} that the free group $\mathbb F_2$ is
C*-simple and has the unique trace property. More general cases where handled using dynamical properties of the Furstenberg boundary: a
discrete group $G$ is C*-simple iff its action on the Furstenberg boundary $\partial_FG$
is free \cite{kk} and has
the unique trace property iff this action is
faithful \cite{bkko}, and there are faithful, non free actions on the Furstenberg boundary \cite{l}. This suggests that ``non triviality''
of the action on the Furstenberg boundary $\partial_FG$ (such as  being faithful or free) somehow measures the ``distance'' of $G$ from being amenable (c.f. \cite[Introduction]{bk}).

In an attempt to formulate the relation between ``boundaries'' with rigidity properties of the C*-algebra generated by the range of arbitrary representations, Bearden and Kalantar  associated a boundary $\mathcal B_u$ to  a unitary representation $u$  of $G$. They  related  properties of $\mathcal B_u$ (and action of $G$ on that) to rigidity properties of the C*-algebra $C_u^*(G)$ generated by the range of $u$. Bearden and Kalantar defined $\mathcal B_u$ as a
``relative'' Hamana $G$-injective envelop of the inclusion $\mathbb C{\rm id}_H\subseteq B(H_u)$,  naturally identified
with a $G$-invariant subspace of $B(H_u)$ and used the Choi-Effros product to make $\mathcal B_u$ into a C*-algebra. Back to the discussion of the previous paragraph, since amenability of $G$ characterizes by the existence of a translation invariant unital positive linear map: $\ell^\infty(G)\to \mathbb C$, given an arbitrary unitary representation $u$ of $G$, it is natural to seek for a ``minimal''
unital $G$-equivariant projection: $B(H_u)\to B(H_u)$, where now positivity is naturally replaced by complete positivity.
This is inline with the basic idea of Hamana in constructing the $G$-boundary \cite[Theorem 3.11]{h2} and
the idea of Bekka to define amenability of unitary representations \cite{b}. By such an analogy, Bearden and Kalantar defined $\mathcal B_u$ as the range of a minimal u.c.p. $G$-equivariant
projection.

In this paper, we introduce and study a generalization of the notion of the Furstenberg boundary of a discrete group $G$ to the setting of a general C*-dinamical system $(G,A,\alpha)$. We define the ``Furstenberg-Hamana boundary'' $\mathcal B_{\pi,u}$ of a given covariant representation $(\pi, u, H)$ of  $(G,A,\alpha)$ as a $G$-invariant subspace of $B(H)$ that carries a canonical C*-algebra structure. Our construction adapts that of Bearden and Kalantar and our boundary generalizes theirs. In many natural cases, including the case of Koopman representation, the Furstenberg-Hamana boundary of $G$ would be commutative, but it is non-commutative in general. We study various properties of this boundary and give some applications in C*-crossed products.

The motivation for this construction is three fold: first that most of the motivating examples of Bearden-Kalantar come from group actions (see for instance examples in \cite[section 3]{bk}) and are already related to a C*-dynamics (which is somehow overlooked in \cite{bk}). Second, the Kalantar-Kennedy affirmation of the Ozawa conjecture already contains the construction of a crossed product and shows that C*-dynamics have a significant role there. Third, the problems of simplicity and unique trace  are already of great importance for crossed product C*-algebras. One could add a rather methodological comment to these by reminding that recently an ``injective'' crossed product is introduced (c.f. \cite{bew}, \cite{bew2}, \cite{bew3}) which could eventually be related to the boundary in this more general sense.

The organization of the paper is as follows: In the next section we recall basic facts about  crossed products and injectivity. In Section \ref{bdry} we define the Furstenberg-Hamana boundary of a covariant representation, study its basic properties, and give several examples. We also find conditions for ``triviality'' of the boundary or its $G$-action.  Section \ref{am} discusses amenability of action on the boundary or that of representations involved. In Section \ref{e} we consider the problem of extending the $G$-action on the boundary to actions of certain subgroup of automorphisms of $G$. Finally, Section \ref{tr} discusses  rigidity properties of traces on crossed product C*-algebras and its relation to the new boundary. We postpone the study of simplicity to a forthcoming paper. In the whole paper we work only with discrete groups.

\section{Preliminaries}
Let $H$ be a complex Hilbert space. A closed self-adjoint subspace $V$ of $B(H)$ is an {\it operator system} if it contains the identity. A linear map $\phi: V\to W$ between operator systems completely positive (c.p.) if all of its amplifications $\phi^{(n)}: \mathbb M_n(V)\to\mathbb M_n(W)$ are positive in the canonical operator space structure on the matrix algebras in the domain and range, and completely
isometric (c.i.) if these are all isometric. The unital completely positive (u.c.p.), unital completely
isometric (u.c.i),  and contractive completely positive (c.c.p.) maps are defined naturally. The c.c.p. maps are known to be completely contractive (c.c) as well. For a discrete group $G$, an operator system $V$ is a $G$-operator system if there is an action of $G$ on $V$, that is, a group homomorphism from $G$ to the group of all bijective u.c.i. maps on $V$.
A linear map  $\phi: V\to W$ between operator systems is a $G$-map if it is
u.c.p. and $G$-equivariant. A $G$-embedding is a c.i. $G$-map. When $W\subseteq V$, a $G$-projection is an idempotent $G$-map.

For a u.c.p. map $\phi : A\to  B$ between C*-algebras, the {\it multiplicative domain} of $\phi$ is
the space
$$\{a \in A : \phi(aa^*) = \phi(a)\phi(a^*) \ {\rm and}\ \phi(aa^*) = \phi(a)\phi(a^*)\}.$$
For $a$ in the multiplicative
domain of $\phi$, we have $\phi(ab) = \phi(a)\phi(b)$ and $\phi(ba) = \phi(b)\phi(a)$, for all $b\in B$ \cite[Theorem 3.18]{p}.

Let $\mathcal C$ be a subcategory of the category of operator systems and u.c.p. maps. An object
$V$ in $\mathcal C$ is injective  if for every c.i. morphism $\iota: U \to W$ and every morphism $\psi : U \to V$ there is a morphisms $\tilde\psi : W \to V$ with $\tilde\psi\circ\iota=\psi$. We are particularly interested in the subcategory of $G$-operator
systems and $G$-maps. Hamana has shown that if $V$ is injective operator system then  the  $G$-operator system $\ell^\infty(G, V)$ is $G$-injective \cite[Lemma 2.2]{h2}.

Let $V$ be a $G$-operator systems. A $G$-extension of $V$ is a pair $(W , \iota)$ consisting of
a $G$-operator system $W$ and a c.i. morphism $\iota : V \to W$.
A $G$-extension $(W , \iota)$ of $V$ is $G$-injective if $W$ is $G$-injective. It is $G$-essential if
for every morphism $\phi : W\to U$ such that $\phi\circ\iota$ is c.i. on $V$, $\phi$ is c.i. on $W$. It is $G$-rigid if for every morphism $\phi: W\to W$ such that
$\phi\circ\iota=\iota$ on $V$, $\phi$ is the identity map on $W$.

A $G$-extension of $V$ which is both $G$-injective and $G$-essential is called  a $G$-injective envelope of $V$. Every  $G$-injective envelope of $V$ is known to be $G$-rigid \cite[Lemma 2.4]{h}.

Every $G$-operator system $V$ has an injective $G$-injective
envelope, denoted by $I_G(V)$, which is unique up to c.i. $G$-isomorphism. A unital $G$-C*-algebra is $G$-injective as a C*-algebra if and only if it is $G$-injective as a $G$-operator system. Since $\ell^\infty(G, B(H))$ is $G$-injective,  if $V \subseteq B(H)$ is a $G$-operator system, then we have a $G$-embedding of $I_G(V)$  into $\ell^\infty(G, B(H))$ and a
$G$-projection: $\ell^\infty(G, B(H))\to I_G(V)$, which makes $I_G(V)$ into a C*-algebra via the Choi–Effros product.

Let $A$ be a C*-algebra and $V\subseteq A$ be an operator subsystem. If there is a surjective idempotent u.c.p. map $\psi : A \to V$, then the
Choi-Effros product $a \cdot b = \psi(ab)$, turns $V$ into a C*-algebra, which is unique up to isomorphism (does
not depend on the map $\psi$) \cite{ce}.

Any action of $G$ on an operator system  $V$ induces an adjoint action on the weak*-compact convex
set of states on $V$. For a state $\nu$ on $V$, the corresponding Poisson map is a unital positive $G$-map from
$V$ to $\ell^\infty(G)$, defined by $$P_{\nu}(a)(g)=\nu(g^{-1}a),$$ for $g\in G$ and $a\in V$.

Let $G$ act on a C*-algebra
$A$ by automorphisms, given by a group homomorphism $\alpha: G\to$ Aut$(A)$. Let $A[G]$ be the *-algebra of formal finite sums $\sum a_tt$,  where $t\mapsto a_t$ is a map from $G$ into $A$ with finite support, under the operations governed by the rules
$$(at)(bs) = a\alpha_t(b)ts, \ (at)^* = \alpha_{t^{-1}}(a)t^{-1},$$
for $a, b \in A$ and $s, t\in G$.
For the C*-dynamical system $(G, A,\alpha)$, the notion of unitary representations
is replaced by that of covariant representations.

A {\it covariant representation} of $(G, A,\alpha)$ is a triple $(\pi, u, H)$ where $\pi$
and $u$ are respectively a representation of $A$ and a unitary representation of $G$ in
the same Hilbert space $H$, satisfying the covariance rule
$$u_t\pi(a)u_t^* = \pi(\alpha_t(a)),$$
for $t\in G$ and $a\in A$. A covariant representation $(\pi, u, H)$ gives rise to a *-homomorphism $\pi\rtimes u$ from $A[G]$
into $B(H)$, defined by
$$(\pi\rtimes u)(\sum a_tt) =\sum \pi(a_t)u_t.$$

The full crossed product $A\rtimes G$ of $(G, A, \alpha)$  is the
C*-completion of $A[G]$ in the norm
$$\Vert a\Vert_{\rm max} = \sup \Vert(\pi\rtimes u)(a)\Vert,$$
where the supremum  runs over all covariant representations $(\pi, u, H)$. Every covariant representation  $(\pi, u, H)$ induces a representation
of $A\rtimes G$, again denoted by $\pi\rtimes u$. Conversely, every non
degenerate representation of $A\rtimes G$ comes in this way from a covariant representation. This means that $A\rtimes G$  is the universal C*-algebra describing the covariant
representations of $(G, A, \alpha)$.

The analog of left regular representation for C*-dynamics is the induced covariant representation. Let $\sigma$ be a representation of $A$ on a Hilbert space $H_\sigma$ and
set $K = \ell^2(G, H_\sigma) = \ell^2(G) \otimes H_\sigma$. The covariant representation $(\tilde\sigma, \lambda, K)$ with
$$\tilde\sigma(a)\xi(t) = \sigma(\alpha_{t^{-1}}(a))\xi(t),\ \ \lambda_s\xi(t) =\xi(s^{-1}t),$$
for $\xi\in K$, is called the  covariant representation induced by $\sigma$, and is denoted by Ind$(\sigma)$. The reduced crossed product $A\rtimes_{\rm red} G$ is the C*-completion of $A[G]$ in the norm
$$\Vert a\Vert_{\rm red} = \sup \Vert{\rm Ind}(\sigma)(a)\Vert,$$
where the supremum runs over all representations $\sigma$ of $A$.
When $(\sigma, H_\sigma)$ is a faithful representation of $A$ one has $\Vert a\Vert_{\rm red} =\Vert$Ind$(\sigma)(a)\Vert$, for
all $a\in A[G]$, and therefore $A\rtimes_{\rm red}  G$ is faithfully represented into $\ell^2(G)\otimes H_\sigma$. Also there is a canonical surjective homomorphism from $A\rtimes_{\rm max} G$ onto $A\rtimes_{\rm red}  G$.

\section{boundary of covariant representations}\label{bdry}
Suppose that $ (G,A,\alpha) $ is a C*-dynamical system and $ (\pi,u, H) $ is a covariant representation. Let $A^\alpha:=\{a\in A: \alpha_t(a)=a\ (t\in G)\}$ be the fixed point algebra of the action. We are going to replace the relative $G$-injective envelop of the inclusion $\mathbb C{\rm id}_H\subseteq B(H_u)$ in \cite{bk} with the relative $G$-injective envelop of the inclusion $\pi(A^\alpha)\subseteq B(H_u)$.
We define
$$ \mathfrak M=\lbrace \phi:B(H)\rightarrow B(H):  \phi\ {\rm is\ a} \ G{\rm -map\ and}\ \phi={\rm id \ on\ } \pi(A^\alpha) \rbrace. $$

\begin{lemma}\label{min}
The set of $\mathfrak M $ with the partial order
$$ \phi\leq\psi \Longleftrightarrow \Vert \phi(x) \Vert \leq \Vert \psi(x) \Vert $$
contains a minimal element.
\end{lemma}
\begin{proof}
First note that $\mathfrak   M $ is nonempty because it contains the identity. Let us observe that any decreasing net $ (\phi_{i})\subseteq\mathfrak  M $  has a lower bound. Since
$\mathfrak   M $ is compact in point-weak* topology, there exists a subnet $ (\phi_{j}),$ and an element $ \phi_{0} \in\mathfrak   M $ such that $ \phi_{ij}(x) \stackrel{w^{*}}{\longrightarrow} \phi_{0}(x), $ for each $x \in B(H). $
 Then
 $$ \Vert \phi_{0}(x) \Vert \leq \limsup_{j} \Vert \phi_{j}(x) \Vert = \inf_{i} \Vert \phi_{i}(x) \Vert,$$
 for each $x \in B(H)$, proving the claim.
 Now $\mathfrak   M $ has a minimal element by the Zorn lemma.
\end{proof}

The following result is essentially proved in \cite{bk}. We reproduce their proof here for the sake of completeness.

\begin{proposition} \label{b}
Suppose $ \phi_{0} $ is a minimal element of $\mathfrak M $. Then

$(i)$
$ \phi_{0}  $ is an idempotent.

$(ii)$
($u$-essentiality) Every $G$-map $ \psi: {\rm Im}(\phi_{0}) \rightarrow B(H) $ with $\psi={\rm id}$ on $\pi(A^\alpha)$, is isometric.

$(iii)$
($u$-minimality) ${\rm Im}(\phi_{0}) \subseteq B(H) $ is minimal among subspace of $ B(H) $ containing $\pi(A^\alpha)$ that are images of $G$-projections.

$(iv)$
($u$-rigidity) The identity map is the unique $G$-map on $ {\rm Im}(\phi_{0})$ which acts as identity on $\pi(A^\alpha)$.

$(v)$
($u$-injectivity) If $ X \subseteq Y $ are $G$-invariant subspaces of $ B(H) $ and $ \psi: B(H) \rightarrow B(H) $ is a $G$-map such that $ \psi(X) \subseteq {\rm Im}(\phi_{0}) $ then there is a $G$-map $ \tilde{\psi}: B(H) \rightarrow B(H) $ such that $ \hat{\psi}(Y) \subseteq {\rm Im}(\phi_{0}) $ and $ \tilde{\psi}=\psi $ on $X$.

\end{proposition}

\begin{proof}
(i) Put
$ \phi^{(n)}=\frac{1}{n} \sum_{k=1}^{n} \phi_{0}^{k} $, then
each $\phi_0^k$ is a $G$-map satisfying $\phi_0^k={\rm id}$ on $\pi(A^\alpha)$ and $\phi^{(n)}\leq \phi_0$, therefore
by minimality of $ \phi_{0} $,  $ \Vert \phi^{(n)}(x) \Vert = \Vert \phi_{0}(x) \Vert $, for each $ x \in B(H) $. Hence
\begin{align*}
\Vert \phi_{0}(x)-\phi_{0}^{2}(x) \Vert &= \Vert \phi_{0}(x-\phi_{0}(x)) \Vert = \Vert \phi^{(n)}(x-\phi_{0}(x)) \Vert \\&= \frac{1}{n} \Vert \phi_{0}(x)-\phi_{0}^{n+1}(x) \Vert   \leq \frac{2}{n} \Vert x \Vert,
\end{align*}
for each $ x \in B(H) $. Thus $ \phi_{0}=\phi_{0}^{2} $.
\\
(ii) Let $ \psi:{\rm Im}(\phi_{0}) \rightarrow B(H) $ be a $G$-map  with $\psi={\rm id}$ on $\pi(A^\alpha)$,  then $ \psi\circ \phi_{0} \in\mathfrak  M $ and for any $ x \in {\rm Im}(\phi_{0}) $,
\begin{align*}
\Vert x \Vert =\Vert \phi_{0}(x) \Vert \leq \Vert \psi(\phi_{0}(x)) \Vert = \Vert \psi(x) \Vert \leq \Vert x \Vert.
\end{align*}
(iii) Suppose $ \phi :B(H)\rightarrow B(H) $ is a $G$-projection with $\pi(A^\alpha) \subseteq {\rm Im} (\phi) \subseteq {\rm Im}(\phi_{0}) $, then by (ii), $ \phi$ is isometric on ${\rm Im}(\phi_{0})$, hence
$$  \Vert \phi_{0}(x)-\phi(\phi_{0}(x)) \Vert = \Vert \phi\big(\phi_{0}(x)-\phi(\phi_{0}(x))\big) \Vert =0,  $$
and so $ \phi=$id on ${\rm Im}(\phi_{0}) $. \\
(iv) Let $ \psi:{\rm Im}(\phi_{0}) \rightarrow {\rm Im}(\phi_{0}) $ be a $G$-map acting as id on $\pi(A^\alpha)$. By (ii), $ \psi $ is isometric and for any $ x\in B(H) $,
$ \Vert \psi \circ \phi_{0}(x) \Vert \leq \Vert \phi_{0}(x) \Vert $, therefore $  \psi\circ \phi_{0}$ is  minimal in $\mathfrak  M $ and so by (i), a $G$-idempotent. For each $ x \in B(H) $, we have
\begin{align*}
\Vert \phi_{0}(x) - \psi \circ \phi_{0}(x) \Vert &= \Vert \psi\big( \phi_{0}(x) -\psi \circ \phi_{0}(x) \big) \Vert \\  &= \Vert \psi \circ \phi_{0}(x)-(\psi o \phi_{0})^{2}(x) \Vert =0
\end{align*}
 that is, $ \psi=$id on its domain.\\
(v) Let $ X \subseteq Y $ are $G$-invariant subspaces of $ B(H) $ and $ \psi: B(H) \rightarrow B(H) $ is a $G$-map such that $ \psi(X) \subseteq {\rm Im}(\phi_{0}) $. Put $ \tilde{\psi}=\phi_{0}\circ\psi $. Then  $ \tilde{\psi} $ is a $G$-map, $ \tilde{\psi}(Y) \subseteq {\rm Im}(\phi_{0}) $ and $ \tilde{\psi}=\psi $ on $X$, since
$ \tilde{\psi}(x) =\phi_{0}(\psi(x))=\psi(x) $, for $x\in X$. This means that we have the following commutative diagram
\begin{center}
$\xymatrix
{ B(H)\ar@{-->}[drr]^{\tilde{\psi}}&& \\ 
	B(H) \ar@{->}[rr]_{\psi} \ar@{<-}[u]^{\phi_0} && B(H)
}
$

\end{center}

as required.
\end{proof}

\begin{proposition} \label{unique}
The image of a minimal element of $\mathfrak M $ is unique up to isomorphism.
\end{proposition}
\begin{proof}
Let $ \phi,\psi :B(H) \rightarrow B(H) $ be minimal elements of $\mathfrak  M $. By part (iv) of the above proposition $ \phi \circ \psi $ is identity on $ {\rm Im}(\phi) $ and   $ \psi \circ \phi $ is identity on $ {\rm Im}(\psi) $. Now the restriction of $ \psi$ to ${\rm Im}(\phi)$ is the required $G$-isomorphism onto ${\rm Im}(\psi)$.
\end{proof}

\begin{definition}
Suppose $ \phi $ is a minimal element of $\mathfrak M $. Then ${\rm Im}( \phi) $ is a $G$-C*-algebra under the Choi-Effros product, called the
Furstenberg–Hamana boundary (or simply boundary) of the covariant representation $ (\pi, u, H)$, and is denoted by
$\mathcal{B}_{\pi,u}$.
\end{definition}

By construction, $\pi(A^\alpha)\subseteq \mathcal{B}_{\pi,u}$. Since the Choi-Effros product induced by $\phi$ is the same as the usual product on $\pi(A^\alpha)$, $\pi(A^\alpha)$ is indeed a C*-subalgebra of $\mathcal{B}_{\pi,u}$.

Let us examine the boundary in some canonical examples. When $G$  acts via $\alpha$ on a discrete set $S$, considered as a measure space
with the counting measure and $\lambda^S: G\to B(\ell^2(S))$ is  the corresponding Koopman
representation $\lambda^S_g(\delta_s) = \delta_{gs}$, for $g\in G$ and $s\in S$, and $M:\ell^\infty(S)\to B(\ell^2(S))$ is  the multiplication represetation $M_f\xi=f\xi$, for $f\in\ell^\infty(S), \xi\in \ell^2(S)$, then $(M, \lambda^S, \ell^2(S))$ is a covariant representation. Since the map $\phi: B(\ell^2(S))\to \ell^\infty(S)$ defined by $\phi(x)(s)=\langle x\delta_s,\delta_s\rangle$ is a $G$-projection, $\mathcal{B}_{M,\lambda^S}$ is also characterized as the image of a minimal $G$-projection on $\ell^\infty(S)$. Hence there must exist a compact $G$-space $\partial_\alpha(G,S)$ such
that $\mathcal{B}_{M,\lambda^S} = C(\partial_\alpha(G,S))$. Bearden and Kalantar also conclude that  $\mathcal{B}_{\lambda^S} = C(\partial(G,S))$, for some compact $G$-space $\partial(G,S)$. The space $\partial(G,S)$  is trivial if and only if the action is amenable
in the sense of Greenleaf \cite{g}, that is, there is an invariant mean on $\ell^\infty(S)$. Also $\partial(G,S) = \partial_F(G)$ if and only if each pint $s\in S$ has an amenable stablizer \cite[Example 3.9]{bk}. These $G$-spaces are not the same in general. If there is a $G$-projection   $\phi_0: \ell^\infty(S)\to \ell^\infty(S)^\alpha=\ell^\infty(S/G)$, then by the above argument, its range is the same as the boundary, i.e., $\mathcal{B}_{M,\lambda^S} =\ell^\infty(S/G)$. On the other hand, the composition of $\phi_0$ with any state of $\ell^\infty(S/G)$ is an invariant mean on $\ell^\infty(S)$ (since $G$ acts trivially on the orbit space $S/G$), and so in this case, $\mathcal{B}_{\lambda^S} =\mathbb C$. If the action is not transitive, the boundary we defined takes a non trivial value (encoding some dynamical data), while the Bearden-Kalantar boundary is trivial.

As another example, let $\lambda^G$ and $\rho^G$ be the left and right regular representation of $G$ and note that the ranges of these representations commute. In particular, $u_g:=\lambda_g^G\rho_g^G$ defines a representation of $G$, and if $C_u^*(G)$ is the $C^*$-algebra generated by the range of $u$ in $B(\ell^2(G))$, then $G$ acts on $C_u^*(G)$ by $\alpha:=$Ad$_{\lambda^G}$. If $\iota$ is the inclusion map of $C_u^*(G)$ in $B(\ell^2(G))$, then $(\iota,\lambda^G, \ell^2(G))$ is a covariant pair, and since the Dirac function $\delta_e$ at the neutral
element $e\in G$ is an invariant vector for $u$, the corresponding vector functional
is an invariant state on $B(\ell^2(G))$. Thus $\mathcal{B}_{\lambda^G}=\mathbb C$ is trivial.  On the other hand, if $G$ is not an ICC group, there is $g\neq e$ with finite conjugacy class. Then $u_e\neq x:=\sum_{h\in G} u_{hgh^{-1}}\in C_u^*(G)^\alpha$, thus  $\mathcal{B}_{\iota, \lambda^G}\neq \mathbb C$. Next, since $\delta_e$ is invariant under $u$ (and along with its scalar multiples form the set of all invariant vectors of $u$), so is $\ell^2(G)\ominus\mathbb C\delta_e=\ell^2(G\backslash\{e\})$, and  if $u_0$ is the subrepresentation of $u$ obtained by restricting to this subspace (which  is  the same as the Koopman representation associated to
the action of $G$ on $S = G\backslash\{e\}$ by conjugation and is unitarily equivalent to the direct
sum of quasi-regular representations associated to stabilizer subgroups), then for
each $s\in S$, the stabilizer subgroup at $s$  is the same as the centralizer of $s$ in $G$. When all of these stabilizer subgroups are amenable (like when $G=\mathbb F_2$ is the free group on two generators) then for the inclusion $\iota_0$ of $C_{u_0}^*(G)$ in $B(\ell^2(S))$ we have the covariant pair $(\iota_0, \lambda^G, \ell^2(S))$ with $\mathcal{B}_{ \lambda^G}=C(\partial_FG)$ \cite[Example 3.10]{bk}, whereas $\mathcal{B}_{\iota_0, \lambda^G}$ is noncommutative: since $G$ acts on $C_{u_0}^*(G)$ by an inner action (namely, by  $\alpha:=$Ad$_{\lambda^G}$), which is exterior equivalent to the trivial action, the fixed point algebra $C_{u_0}^*(G)^\alpha$ is isomorphic with $C_{u_0}^*(G)$ (c.f., \cite[Remark 3.9]{pi}), and  $\mathcal{B}_{\iota_0, \lambda^G}$ should contain a copy of $C_{u_0}^*(G)$, which is noncommutative, for instance for $G=\mathbb F_2$.

As the third example, recall that a non-singular action of $G$   on a probability space
$(X, \nu)$ is Zimmer-amenable if there is a $G$-equivariant unital positive projection: $\ell^\infty(G)\bar\otimes L^\infty(X, \nu)\to L^\infty(X, \nu)$  \cite{z}. Since $G$ is  discrete, this is equivalent to $G$-injectivity of $L^\infty(X, \nu)$  \cite[Remark 2.3]{h2}. When we have such an amenable action, there is a $G$-equivariant
embedding: $C(\partial_FG)\hookrightarrow L^\infty(X, \nu)$ and a  $G$-equivariant projection: $ L^\infty(X, \nu) \to C(\partial_FG)$. Thus $\mathcal{B}_{\lambda^X}=C(\partial_FG)$ \cite[Example 3.11]{bk}. On the other hand, for the covariant pair $(M, \lambda^X, L^2(X,\nu))$, with $M$ and $\lambda^X$ being the multiplication representation of $L^\infty(X,\nu)$ and the Koopman representation
associated to the action of $G$ on $X$, both in $L^2(X,\nu)$, since $L^\infty(X, \nu)$ is  $G$-injective,  $\mathcal{B}_{M,\lambda^X}=L^\infty(X, \nu)^G=L^\infty(X/G, \pi_*\nu)$, where $\pi_*\nu$ is the pullback of $\nu$ along the quotient map $\pi: X\to X/G$.

One expects that ``triviality'' of the boundary should be equivalent to some sort of amenability
property. In fact, it is shown by Bearden and Kalantar \cite[Proposition 3.12]{bk} that the boundary $\mathcal B_u$ of a unitary representation $u$ is trivial (that is, isomorphic to $\mathbb C$) if and only if $u$
is an amenable representation, in the sense of Bekka \cite{b}. In the case of covariant representations of a C*-dynamical system  $ (G,A,\alpha)$, the fixed point algebra plays the role of ``trivial'' boundary, and the notion of $G$-injectivity of Buss-Echterhoff-Willett is the analog of the amenability of Bekka.
Let us recall the notion of $G$-injectivity for $G$-C*-algebras \cite[Definition 2.1]{bew}.

\begin{definition}
An equivariant ccp map $\phi:A \to B$ between $G$-C*-algebras is $G$-injective if
for any equivariant injective $*$-homomorphism  $\iota: A \to C$ there is  an equivariant contractive completely positive (ccp) map  $\tilde\phi:C \to B$ such that the following diagram:

\[
\begin{tikzcd}[column sep=small]
	A \arrow{r}{\iota}  \arrow{rd}{\phi}
	& C \arrow{d}{\tilde\phi} \\
	& B
\end{tikzcd}
\]

\noindent commutes.
A $G$-C*-algebra $B$ is $G$-injective if any  ccp $G$-map: $A \to B$ is $G$-injective.

\end{definition}

Following the basic idea of ``relative'' injectivity in \cite{p2}, if in the above definition, we require that all the morphisms are ccp and the objects $A$, $B$, and $C$ are operator systems included in $B(H)$, for some Hilbert space $H$, we get the notion of $G$-injectivity of a $G$-map: $A\to B$ relative to $B(H)$. In this case, we say that $B$ is $G$-{\it injective relative to} $B(H)$ (or, the inclusion $B\subseteq B(H)$ is $G$-injective) if  for each $A\subseteq B(H)$, any ccp $G$-map: $A \to B$ is $G$-injective relative to $B(H)$. For a representation $u: G\to B(H_u)$, we say that a $G$-operator system $A\subseteq B(H_u)$ is $u$-{\it injective}, if the inclusion $A\subseteq B(H_u)$ is $G$-injective (compare with Proposition \ref{b}$(v)$).

The notion of amenable unitary representation is introduced by Bekka \cite[Definition 1.1]{b}. A unitary representation $u: G\to B(H)$ is amenable if there is a unital positive $G$-projection
$\phi : B(H) \to \mathbb C$, where $G$ acts on $B(H)$ by $Ad_u$. This is to say that the inclusion $\mathbb C1\subseteq B(H)$ is $G$-injective. In this case, the Bearden-Kalantar boundary is trivial, that is,  $ \mathcal{B}_{u}\cong \mathbb C$.

We say that the covariant representation $(\pi, u, H)$ is amenable if the inclusion $\pi(A^\alpha)\subseteq B(H)$ is $G$-injective, or equivalently, there is a $G$-projection: $B(H)\to \pi(A^\alpha)$. We have the following characterizations of ``trivial'' boundary in our setting.

 \begin{proposition} \label{trivial}
 	Let $(\pi, u, H)$ be a covariant representation of the C*-dynamical system  $ (G,A,\alpha)$. Then
 	 $ \mathcal{B}_{\pi,u}= \pi(A^\alpha)$ iff $(\pi, u, H)$ is amenable.
 \end{proposition}
 \begin{proof}
If $\pi(A^\alpha)$ is $G$-injective then there is a $G$-projection $\phi_0: B(H)\to \pi(A^\alpha)$, which is then clearly a minimal element of the set $ \mathfrak M$ above. The converse is trivial, as the inclusion $ \mathcal{B}_{\pi,u}\subseteq B(H)$ is $G$-injective (i.e., $ \mathcal{B}_{\pi,u}$ is $u$-injective).

 \end{proof}

Another aspect of ``triviality'' is the triviality of action. The next result which extends \cite[Proposition 3.17]{bk} takes care of trivial actions.

\begin{proposition} \label{trivial2}
	Every element of $G$ with finite conjugacy class acts trivially on $ \mathcal{B}_{\pi,u}$  for any
	covariant representation $(\pi, u, H)$.
\end{proposition}
\begin{proof}
	Let $F=$Conj$(g)$ be a finite conjugacy class of some element $g$. Then $\phi:=\frac{1}{|F|} \sum_{h\in F} {\rm Ad}_{u_h}$ is a $G$-map on $ \mathcal{B}_{\pi,u}$,
	hence $\phi$ is the identity by $u$-rigidity. Since projections are extreme points of the positive part
	of the unit ball of a C*-algebra, $u_hpu_h^* = p$, for every $h\in F$ and every projection $p \in \mathcal{B}_{\pi,u}$.
	Since an injective C*-algebra is generated by its projections \cite[Lemma 3.16]{bk}, Ad$_{u_h}$ is the identity map on $\mathcal{B}_{\pi,u}$, for every $h\in F$.
\end{proof}

It is natural to expect $ \mathcal{B}_{\pi,u}$ to behave naturally under the restriction and induction of representations. This is indeed the case by next result, where in the induction case we use weak containment assumptions (see \cite{d} for details).

First we need to know how $ \mathcal{B}_{\pi,u}$ behaves under weak containment of representations. If  $(\pi, u, H)$ is  a covariant representation of the C*-dynamical system $ (G,A,\alpha) $, we denote the C*-completion of the range of $\pi\rtimes u$ on the algebraic crossed product $A\rtimes_{\rm alg} G$ by $A\rtimes_{\pi, u} G$. We say that  a covariant representation $(\pi, u, H)$ is {\it weakly contained} in another  covariant representation $(\sigma, v, K)$, and write $(\pi, u, H)\preceq (\sigma, v, K)$ if the identity map on $A\rtimes_{\rm alg} G$ extends to a *-homomorphism: $ A\rtimes_{\sigma, v} G\to A\rtimes_{\pi, u} G$, and in this case, refer to  the latter as the canonical map between these C*-algebras. We say that   $(\pi, u, H)$ and $ (\sigma, v, K)$ are weakly equivalent if each is weakly contained in the other.

\begin{lemma} \label{wc}
	Let	$(\pi, u, H)$ and $(\sigma, v, K)$ be covariant representations of $ (G,A,\alpha)$, then there is a
	$G$-map from $\mathcal{B}_{\sigma,v}$ to $\mathcal{B}_{\pi,u}$  if and only if there is a $G$-map from $B(K)$ to $B(H)$.  In particular,
	if $(\pi, u, H)$  is weakly contained in  $(\sigma, v, K)$,  there is a $G$-map: $ \mathcal{B}_{\sigma,v}\to\mathcal{B}_{\pi,u}$.
\end{lemma}
\begin{proof}
	Since there are $G$-maps $\mathcal{B}_{\pi,u}\to B(H) \to \mathcal{B}_{\pi,u}$ and $\mathcal{B}_{\sigma,v}\to B(K) \to \mathcal{B}_{\sigma,v}$, we get the first assertion. If $(\pi, u, H)\preceq (\sigma, v, K)$, then the canonical *-homomorphism: $ A\rtimes_{\sigma, v} G\to A\rtimes_{\pi, u} G$ extends to a u.c.p. map: $B(K) \to B(H)$, which   has to be $G$-equivariant, as $A\rtimes_{\sigma, v} G$ is in its multiplicative domain.
	
\end{proof}

The above lemma shows that the boundary $\mathcal{B}_{\pi,u}$ of a covariant representation $(\pi, u, H)$ is determined by the weak equivalence class of $(\pi, u, H)$, in other words, if covariant representations $(\pi, u, H)$ and $(\sigma, v, K)$ are weakly equivalent, then the $G$-maps between the corresponding boundaries are inverse of each other (by rigidity), and so the boundaries is isomorphic as $G$-C*-algebras.

\begin{proposition} \label{res}
	If $(\pi, u, H)$ is a covariant representation of $(G, A, \alpha)$, $\Lambda\leq G$ and $B$ is a  $G$-invariant C*-subalgebra of $A$, $u_0$ and $\alpha_0$ are the restrictions of $u$ and $\alpha$ on $\Lambda$ and $\pi_0$ is the restriction of $\pi$ on $B$, then
	
	$(i)$ $(\pi, u_0, H)$ and $(\pi_0, u, H)$ are  covariant representations of $(\Lambda, A, \alpha_0)$ and $(G, B, \alpha)$, respectively, and there are $\Lambda$-embedding: $\mathcal{B}_{\pi,u_0}\hookrightarrow \mathcal{B}_{\pi,u}$ and $G$-embedding:  $\mathcal{B}_{\pi_0,u}\hookrightarrow \mathcal{B}_{\pi,u}$,
	
	$(ii)$ if moreover $\Lambda\unlhd G$ and the action of $\Lambda$ on $\mathcal{B}_{\pi,u_0}$ is faithful, then  kernel of the action of $G$ on $\mathcal{B}_{\pi,u}$ lies in the  centralizer of $\Lambda$ in $G$. In particular, if a normal subgroup
	$\Lambda\unlhd G$ has trivial centralizer, the  action of $\Lambda$ on $\mathcal{B}_{\pi,u_0}$ is faithful iff the action
    $G$ on $\mathcal{B}_{\pi,u}$ is so,

    $(iii)$ if $(\pi, u_1, H_1)$ and $(\pi, u_2, H_2)$  are covariant representations, then there are $G$-maps: $\mathcal{B}_{\pi,u_i}\to\mathcal{B}_{\pi,u_1\otimes u_2}$, for $i = 1, 2$,

    $(iv)$ if $u\preceq {\rm ind}_{\Lambda}^G u_0$ and $u_0\preceq v_0\otimes w_0$, for some
    unitary representations $v_0,w_0$ of $\Lambda$, then for $v={\rm ind}_{\Lambda}^G v_0$ there is a $\Lambda$-map: $\mathcal{B}_{\pi,v}\to \mathcal{B}_{\pi,u}$, in particular, if $\Lambda$ is coamenable in $G$, then there is a $G$-isomorphism $\mathcal{B}_{\pi,u}\cong\mathcal{B}_{\pi,{\rm ind}_{\Lambda}^G(u_0)}$,

    $(v)$ if $(\pi_1, u, H)$ and $(\pi_2, u, H)$  are covariant representations so that the ranges of $\pi_1$ and $\pi_2$ commute, then $(\pi_1\times\pi_2, u, H)$ is a covariant representation for the diagonal action of $G$ on $A\otimes_{\rm max} A$  and there are $G$-maps: $\mathcal{B}_{\pi_i,u}\to\mathcal{B}_{\pi_1\times\pi_2,u}$, for $i = 1, 2$,

    $(vi)$ if $(\pi_1, u_1, H_1)$ and $(\pi_2, u_2, H_2)$  are covariant representations, then $(\pi_1\otimes\pi_2, u_1\otimes u_2, H_1\otimes H_2)$ is a covariant representation for the tensor action of $G$ on $A\otimes_{\rm min} A$  and there is a $G$-map: $\mathcal{B}_{\pi_1,u_1}\otimes_{{\rm min}}\mathcal{B}_{\pi_1,u_1}\to\mathcal{B}_{\pi_1\otimes\pi_2,u_1\otimes u_2}$,

    $(vii)$ if $J$ is a $G$-invariant closed ideal  of $A$ and the restriction $\pi_0$  of $\pi$ on $J$  is non degenerate, then there is a $G$-isomorphism $\mathcal{B}_{\pi,u}\cong\mathcal{B}_{{\rm ind}_{J}^A\pi_0,u}$.

\end{proposition}
\begin{proof}
	(i) If $\phi: B(H) \to\mathcal{B}_{\pi,u}$ be a $G$-projection, now considered as a $\Lambda$-map, then by Proposition \ref{b}$(ii)$,  its restriction to $\mathcal{B}_{\pi,u_0}$ is a $\Lambda$-embedding. The same argument works for $\mathcal{B}_{\pi_0,u}$.
	
	(ii) By $(i)$, identify $\mathcal{B}_{\pi,u_0}$ with a $\Lambda$-invariant subspace of $\mathcal{B}_{\pi,u}$.
	If $g$ is in the kernel of the action of $G$ on $\mathcal{B}_{\pi,u}$ and $h\in\Lambda$, for $k:=h^g$, Ad$u_kx=$Ad$u_hx$, for $x\in \mathcal{B}_{\pi,u_0}$, thus $[g,h]$ is in the kernel of the action of $\Lambda$ on $\mathcal{B}_{\pi,u_0}$, which is trivial. Therefore, $g$ is in the  centralizer of $\Lambda$ in $G$. The last assertion now follows.
	
	(iii) If $(\pi_1, u_1, H_1)$ and $(\pi_2, u_2, H_2)$  are covariant representations, then compose the $G$-map: $B(H_1)\to   B(H_1)\bar\otimes B(H_2);\ a\mapsto a\otimes I$, with a $G$-map: $B(H_1)\bar\otimes B(H_2)\to \mathcal{B}_{\pi,u_1\otimes u_2}$, and then restrict to $\mathcal{B}_{\pi,u_1}$. Do the same for $u_2$.
	
	(iv) By \cite[Appendix E, F]{bhv},
	$$ u\preceq {\rm ind}_{\Lambda}^G(u_0)\preceq {\rm ind}_{\Lambda}^G(v_0\otimes w_0)\preceq {\rm ind}_{\Lambda}^G(v_0)\otimes {\rm ind}_{\Lambda}^G(w_0),$$	
	and the result follows from  (iii) and Lemma \ref{wc}.
	
	When $\Lambda$ is coamenable in $G$, for the left regular representation $\lambda^G$ we have
$$u\sim u\otimes 1_G\preceq u\otimes \lambda^G=u\otimes {\rm ind}_{\Lambda}^G(1_G)\preceq {\rm ind}_{\Lambda}^G(u_0\otimes 1_G)={\rm ind}_{\Lambda}^G(u_0).$$
Hence 	 there is a $G$-map $ \mathcal{B}_{\pi,{\rm ind}_{\Lambda}^G(u_0)}\to\mathcal{B}_{\pi,u}$. On the other hand, ${\rm ind}_{\Lambda}^G(u_0)=u\otimes\lambda^{G/\Lambda}$, where the last term is the Koompan representation of the homogeneous space $G/\Lambda$, thus there is a $G$-map $\mathcal{B}_{\pi,u}\to \mathcal{B}_{\pi,{\rm ind}_{\Lambda}^G(u_0)}$.
	Now by rigidity, both
	$G$-maps are isomorphisms.
	
	(v) If $(\pi_1, u, H)$ and $(\pi_2, u, H)$  are covariant representations so that the ranges of $\pi_1$ and $\pi_2$ commute, then $(\pi_1\times\pi_2, u, H)$ is a covariant representation for the diagonal action of $G$ on $A\otimes_{\rm max} A$ (c.f. \cite[Remark 2.27]{w}, \cite[3.3.7]{bo}). Restrict the $G$-map: $B(H)\to \mathcal{B}_{\pi_1\times \pi_2, u}$ to $\mathcal{B}_{\pi_1,u}$ and do the same for $\pi_2$.
	
	(vi) If $(\pi_1, u_1, H_1)$ and $(\pi_2, u_2, H_2)$  are covariant representations, then $(\pi_1\otimes\pi_2, u_1\otimes u_2 , H_1\otimes H_2)$ is a covariant representation for the diagonal action of $G$ on $A\otimes_{\rm min} A$ (c.f. \cite[Remark 2.27]{w}, \cite[3.2.4]{bo}). Since Ad$_{u_1\otimes u_2}=$Ad$_{u_1}\otimes$Ad$_{u_2}$, the canonical embedding $B(H_1)\otimes_{{\rm min}}B(H_2)\hookrightarrow
B(H_1\otimes H_2)$ is a $G$-map. Compose this with $G$-embedding: $\mathcal{B}_{\pi_1,u_1}\otimes_{{\rm min}}\mathcal{B}_{\pi_1,u_1}\hookrightarrow B(H_1)\otimes_{{\rm min}}B(H_2)$ (c.f., \cite[3.6.1]{bo}) and $G$-map: $B(H_1\otimes H_2)\to \mathcal{B}_{\pi_1\otimes\pi_2,u_1\otimes u_2}$. 	

(vii) Since $\pi_0$ is non degenerate, there is a unique representation on $A$ which extends $\pi_0$ \cite[2.10.4]{di}, and the result follows from Lemma \ref{wc}.

\end{proof}

The next result  relates the Furstenberg-Hamana boundary of induced representations to the classical Furstenberg boundary (defined in \cite{h2} and characterized in \cite[Theorem 3.11]{kk}).

\begin{proposition} \label{main1}
	Let   $(\sigma, H_\sigma)$ be a representation of a unital C*-algebra $A$ with induced representation Ind$(\sigma)=(\tilde\sigma, \lambda, \ell^2(G, H_\sigma))$. Then there is a $G$-embedding: $C(\partial_FG)\hookrightarrow\mathcal{B}_{{\rm Ind}(\sigma)}$ and a $G$-projection: $\mathcal{B}_{{\rm Ind}(\sigma)}\twoheadrightarrow C(\partial_FG)$. If $(\pi, u, H)\preceq{\rm Ind}(\sigma)$, for some representation $\sigma$ of $A$, then there is a $G$-map: $C(\partial_FG)\to\mathcal{B}_{\pi,u}$.
\end{proposition}
\begin{proof}
	Consider the trivial C*-subalgebra $B:=\mathbb C\subseteq A$ and note that $\tilde\sigma$ restricts to the trivial representation $tr$ on $B$. Apply Proposition \ref{res}$(i)$, to get a $G$-embedding $$C(\partial_FG)=\mathcal B_{\lambda^G}=\mathcal{B}_{tr,\lambda}\hookrightarrow\mathcal{B}_{\tilde\sigma,\lambda}=\mathcal{B}_{{\rm Ind}(\sigma)},$$
	 where $\lambda^G$ is the left regular representation of $G$ and the first equality is shown in \cite[Example 3.7]{bk}. Now by $G$-injectivity of $C(\partial_FG)$, there is a $G$-projection: $\mathcal{B}_{{\rm Ind}(\sigma)}\twoheadrightarrow C(\partial_FG)$.
	
	 If $(\pi, u, H)\preceq$ Ind$(\sigma)$, by Lemma \ref{wc}, we have a
	$G$-map: $\mathcal{B}_{\tilde\sigma, \lambda}\to \mathcal{B}_{\pi, u}.$ This composed with the $G$-embedding: $C(\partial_FG)\hookrightarrow\mathcal{B}_{{\rm Ind}(\sigma)}$ gives the desired $G$-map: $C(\partial_FG)\to\mathcal{B}_{\pi,u}$.	
\end{proof}

\section{Amenability} \label{am}

In this section we investigate the relation between the Furstenberg-Hamana boundary $\mathcal{B}_{\pi,u}$ and injective crossed products.

\begin{lemma} \label{emb}
	Let	$(\pi, u, H)$  be a covariant representations of $ (G,A,\alpha)$, then there is a
	$G$-inclusion $\mathcal{B}_{\pi,u}\hookrightarrow I(A\rtimes_{\pi, u} G)$.
\end{lemma}
\begin{proof}
	By injectivity, there is a projection $\phi : B(H) \to  I(A\rtimes_{\pi, u} G)$, extending the
	identity map on $A\rtimes_{\pi, u} G$. In particular, $\phi$ is $G$-equivariant. The restriction of
	$\phi$ to $\mathcal{B}_{\pi,u}$ is a $G$-embedding by Proposition \ref{b}$(ii)$.

\end{proof}

Recall that an action of $G$ on a compact space $X$ is {\it topologically amenable}  if there is a net of continuous maps
$m_i : X \to$ Prob$(G)$ such that $\Vert gm_i(x) - m_i(gx)\Vert_1\to 0$, as $i\to\infty$, for every $g\in G$ \cite[Definition 2.1]{d}. Note that the group $G$ is exact if and only if it admits a topologically amenable action
on a compact space \cite{o}. Also for compact $G$-spaces $X$ and $Y$ if there is a continuous $G$-equivariant map:$ Y \to X$ and
$X$ is topologically amenable
then so is $Y$.

An action of $G$ on a unital C*-algebra $B$ is (strongly)
topologically amenable if  the action of $G$ on the Gelfand spectrum of the
center $Z(B)$ of $B$ is topologically amenable (c.f., \cite{d}). A related notion is defined in \cite[Definition 3.4]{bew2}: an action of $G$ on a  C*-algebra $B$ is strongly
 amenable if there is a net  of norm-continuous, compactly supported, positive type functions $\theta_i : G \to ZM(A)$ such that $\|\theta_i(e)\|\leq 1$
for all $i$ (where $e$ is the identity of $G$), and $\theta_i(g)\to 1$,  strictly and uniformly for $g$ in compact
subsets of $G$. The two notions are equivalent when $B$ is commutative (and unital) \cite[Theorem 5.13]{bew2}. Note that in \cite{bew2} the authors also define the notion of amenable action \cite[Definition 3.4]{bew2}, which is in general weaker than strong amenability.

 Following \cite{bk}, for a covariant representation $(\pi, u, H)$, we say that a boundary $\mathcal{B}_{\pi,u}$ is C*-embeddable if there is a *-homomorphic copy of
 $\mathcal{B}_{\pi,u}$ in $B(H)$. In this case, let $\tilde{\mathcal{B}}_{\pi,u}$ be the C*-algebra generated by $\mathcal{B}_{\pi,u}\cup u(G)$
 in $B(H)$. Note that in general, $\mathcal{B}_{\pi,u}$ is only a subspace of $B(H)$, made into a C*-algebra using the Choi-Effros product.

\begin{proposition} \label{ta}
	If $\mathcal{B}_{\pi,u}$ is unital and C*-embeddable and the action of $G$ on $\mathcal{B}_{\pi,u}$ is topologically amenable, then $G$ is an exact group and $(\pi_0, u, H)\preceq{\rm Ind}(\pi_0)$, where $\pi_0$ is the restriction of $\pi$ to the fixed point algebra $A^\alpha$.
\end{proposition}
\begin{proof}	
If the action of $G$ on $\mathcal{B}_{\pi,u}$ is topologically amenable, then, so is its  restriction the center $Z(\mathcal{B}_{\pi,u})$, and so $G$ is  exact \cite[Theorem 7.2]{d}. Also, if $Z(\mathcal{B}_{\pi,u})=C(X)$, then $\mathcal{B}_{\pi,u}$ is a $G$-$C(X)$-algebra and $\mathcal{B}_{\pi,u}\rtimes_{\rm max} G\cong \mathcal{B}_{\pi,u}\rtimes_{\rm red} G$ by \cite[Theorem 5.3]{d}. Now since $\mathcal{B}_{\pi,u}$ is C*-embeddable,
there is a *-homomorphism: $\mathcal{B}_{\pi,u}\rtimes_{\rm red} G\to \tilde{\mathcal{B}}_{\pi,u}$, canonically mapping $\pi(A^\alpha)\rtimes_{\rm red} G=A^\alpha\rtimes_{\tilde\pi_0,\lambda} G$ onto $A^\alpha\rtimes_{\pi_0,u} G$. Thus $(\pi_0, u, H)\preceq{\rm Ind}(\pi_0)$.

\end{proof}

We also have the following characterization of strong amenability of the $G$-action on $\mathcal{B}_{\pi,u}$, which extends \cite[Theorem
1.1]{kk} and follows immediately by \cite[Theorem 7.5]{bew2}.

\begin{proposition} \label{ta2}
The following are equivalent:
	
	$(i)$ G is exact,
	
	$(ii)$ the action of $G$ on $\mathcal{B}_{\pi,u}$ is strongly  amenable,
	
	$(iii)$ the action of $G$ on any $G$-injective C*-algebra $B$ is strongly amenable.
	
\end{proposition}

Following \cite{bk}, for a normal subgroup $\Lambda\unlhd G$ and a unitary representation $(u, H)$ of $G$, we say that $\Lambda$ is $(\pi, u)$-amenable if there is a $G$-map $\phi: B(H)\to u(\Lambda)^{'}$, such that $\phi=$id on $\pi(A^\alpha)$, where the commutant is taken inside $B(H)$ and is known to be $G$-invariant. Note that we always have $\pi(A^\alpha)\subseteq u(G)^{'}\subseteq u(\Lambda)^{'}$. The whole group $G$ is $u$-amenable iff $u$ is an amenable representation iff all normal subgroups of $G$ are $u$-amenable \cite[Proposition 4.2]{bk}. Also a normal
subgroup $\Lambda\unlhd G$ is $\lambda^G$-amenable iff it is amenable \cite[Proposition 4.3]{bk}. We have the following extensions of these results.

\begin{proposition} \label{a}
	Let	$(\pi, u, H)$  be a covariant representations of $ (G,A,\alpha)$, then the following are
	equivalent:
	
	(i) $(\pi, u, H)$ is an amenable covariant representation,
	
	(ii) any normal subgroup $\Lambda\unlhd G$ is $(\pi, u)$-amenable and there is a conditional expectation: $u(\Lambda)^{'}\to \pi(A^\alpha)$,
	
	(iii) $G$ is $(\pi, u)$-amenable and there is a conditional expectation: $u(G)^{'}\to \pi(A^\alpha)$.
	
\end{proposition}
\begin{proof}
	$(i)\Rightarrow (ii)$. If $\Lambda$  is normal subgroup and $\phi : B(H) \to \pi(A^\alpha)$ is  a $G$-projection, then composing $\phi$ with the inclusion $\iota: \pi(A^\alpha)\hookrightarrow u(\Lambda)^{'}$ we get a $G$-map $\iota\circ\phi : B(H) \to u(\Lambda)^{'}$ and restricting $\phi$  to $u(\Lambda)^{'}$ we get a conditional expectation $\phi_0 : B(H) \to \pi(A^\alpha)$.
	
	$(ii)\Rightarrow (iii)$. Trivial.
	
	$(iii)\Rightarrow (i)$. Let $\phi : B(H) \to u(G)^{'}$ be a $G$-map such that $\phi=$id on $\pi(A^\alpha)$, and compose $\phi$ with a conditional expectation $\tau: u(G)^{'}\to \pi(A^\alpha)$, which is also $G$-equivariant, as $G$ acts trivially on both domain and range, then  $\tau\circ\phi : B(H) \to \pi(A^\alpha)$ is  a $G$-projection.
		
\end{proof}

\begin{theorem} \label{ame}
	Let	$\sigma$  be a representations of $A$, then for a normal subgroup $\Lambda\unlhd G$ the following are
	equivalent:
	
	(i)  $\Lambda$ is ${\rm Ind}(\sigma)$-amenable,
	
	(ii) $\Lambda$ is amenable.
	
\end{theorem}
\begin{proof}
	$(ii)\Rightarrow (i)$. Let $\lambda^G$ be the left regular representation of $G$ and ${\rm Ind}(\sigma) =(\tilde\sigma, \lambda, \ell^2(G, H_\sigma))$ be the induced representation. If $\Lambda$ is amenable, there is a $G$-map: $\ell^\infty(G) \to \ell^\infty(G/\Lambda)$, which lifts to a $G$-map: $\ell^\infty(G, B(H_\sigma)) \to \ell^\infty(G/\Lambda, B(H_\sigma))$, as $G$ acts trivially on $B(H_\sigma)$.
Also there is a  $G$-projection: $ B(\ell^2(G, H_\sigma)) \to \ell^\infty(G, B(H_\sigma))$, as $B(H_\sigma)$ is injective, and so $\ell^\infty(G, B(H_\sigma))$ is $G$-injective \cite[Lemma 2.2]{h2}. Composing these maps we get a $G$-map: $$ B(\ell^2(G, H_\sigma)) \to \ell^\infty(G/\Lambda, B(H_\sigma))\subseteq \lambda(\Lambda)^{'},$$ where the last commutant is inside $ B(\ell^2(G, H_\sigma))$.

$(i)\Rightarrow (ii)$. Conversely, suppose $\Lambda$ is ${\rm Ind}(\sigma)$-amenable, and $\phi: B(\ell^2(G, H_\sigma)) \to \lambda(\Lambda)^{'}$ is a
$G$-map. Composing this with  any state of $\lambda(\Lambda)^{'}$ to get a $G$-map: $B(\ell^2(G, H_\sigma))\to \mathbb C$, with range having a trivial $G$-action. Restricting this to $\ell^\infty(G)$ gives a $\Lambda$-invariant map: $\ell^\infty(G)\to \mathbb C$, showing that $\Lambda$  is amenable.
\end{proof}

\begin{lemma} \label{t}
	Let	$(\pi, u, H)$  be a covariant representations of $ (G,A,\alpha)$, then a normal subgroup $\Lambda\unlhd G$ acts trivially on  $\mathcal{B}_{\pi,u}$ iff it is
	$u$-amenable.
\end{lemma}
\begin{proof}
If $\Lambda$  is $u$-amenable and $\phi : B(H) \to u(\Lambda)^{'}$ is  a $G$-map, by Proposition \ref{b}$(ii)$, the restriction of $\phi$ on $\mathcal{B}_{\pi,u}$ is a $G$-embedding. But $\Lambda$  acts trivially on $u(\Lambda)^{'}$, and so on $\mathcal{B}_{\pi,u}$.

Conversely if $\Lambda$  acts trivially on  $\mathcal{B}_{\pi,u}$, then there is a  $G$-embedding: $\mathcal{B}_{\pi,u}\hookrightarrow u(G)^{'}$. Composing this with any $G$-projection: $B(H) \to \mathcal{B}_{\pi,u}$, we get a $G$-map: $B(H)\to u(G)^{'}$.	
		
\end{proof}

In particular, the kernel of the action  of $G$ on $\mathcal{B}_{\pi,u}$ is the unique maximal $u$-amenable
normal subgroup of $G$, which contains all $u$-amenable normal subgroups of $G$. This is called the $(\pi, u)$-amenable radical of
$G$, and is denoted by Rad$_{\pi,u}(G)$ (c.f., \cite[Definition 4.6]{bk}).

Let $V$ be a $G$-operator system. Then the $G$-injective envelope $I_G(V)$ is characterized as a unique maximal
$G$-essential (or minimal $G$-injective) $G$-extension of $V$.
If $I(V)$ is the injective envelope of $V$ as an
operator system, since Aut$V\subseteq$ Aut$I(V)$, one may regard $I(V)$ as a $G$-extension of $V$. Indeed,  $I(V)$
is a $G$-essential $G$-extension of $V$ and $V\subseteq I(V)\subseteq I_G(V)$ as $G$-spaces. Moreover,  $I(V)$ is unique among the $G$-subspaces  of $I_G(V)$ which are the injective envelope of $V$ \cite[Remark 2.6]{h2}.

Recall that a normal subgroup $\Lambda\unlhd G$ is called {\it coamenable} if there is a $G$-state on  $\ell^\infty(G/\Lambda)$. Let	$(\pi, u, H)$  be a covariant representations of $ (G,A,\alpha)$. A normal subgroup $\Lambda\unlhd G$ is called $(\pi,u)$-{\it coamenable} if there
exists a $G$-map: $u(\Lambda)^{'} \to I(\pi(A^\alpha))$. Here $I(\pi(A^\alpha))$ is the injective envelop of $\pi(A^\alpha)$ which is defined in the unital case as usual, and in non-unital case by passing to the minimal unitization (or directly, and observing that the algebra and its minimal unitization have the same injective envelop; c.f., \cite{bp}).

\begin{theorem} \label{t2}
	Let	$(\pi, u, H)$  be a covariant representations of $ (G,A,\alpha)$, then every coamenable
	normal subgroup  $\Lambda\unlhd G$ is $(\pi,u)$-coamenable. The converse is also true when $(\pi, u, H)={\rm Ind}(\sigma)$, for a representation $\sigma$ of $A$.
\end{theorem}
\begin{proof}
Consider the canonical $G/\Lambda$-action on $u(G)^{'}$, since   $\ell^\infty(G/\Lambda, I(\pi(A^\alpha)))$ is $G$-injective \cite[Lemma 2.2]{h2}, there is a $(G/\Lambda)$-map: $B(H)\to \ell^\infty(G/\Lambda, I(\pi(A^\alpha)))$, whose restriction clearly gives a $G$-map $\phi: u(G)^{'}\to \ell^\infty(G/\Lambda, I(\pi(A^\alpha)))$. By coamenability of $\Lambda$, there is a $G$-map $\psi : \ell^\infty(G/\Lambda)\to\mathbb C$. Tensoring with the identity map on $I(\pi(A^\alpha))$, we get a $G$-map $\psi\otimes {\rm id} : \ell^\infty(G/\Lambda, I(\pi(A^\alpha)))\to I(\pi(A^\alpha))$. Thus we have a $G$-map $(\psi\otimes {\rm id})\circ\phi : u(G)^{'}\to I(\pi(A^\alpha))$, as required.

Conversely, if $(\pi, u, H)={\rm Ind}(\sigma)$, for a representation $\sigma$ of $A$,	then identifying $\ell^\infty(G/\Lambda, I(\sigma(A^\alpha)))$ with the subalgebra of elements in $\ell^\infty(G, I(\sigma(A^\alpha)))$ which are constant on cosets of $\Lambda$, we may regard  $\ell^\infty(G/\Lambda, I(\sigma(A^\alpha)))$ as a $G$-invariant
C*-subalgebra of $B(\ell^2(G, H_\sigma))$. Under this identification,
the fact that $\Lambda$ acts trivially on $\ell^\infty(G/\Lambda, I(\sigma(A^\alpha)))$ implies that $\ell^\infty(G/\Lambda, I(\sigma(A^\alpha)))$  is mapped by $\phi$ into $u(G)^{'}$. Now if  $\Lambda$ is $(\tilde\sigma,\lambda)$-coamenable, there
is a $G$-map $\psi: \lambda(\Lambda)^{'} \to I(\sigma(A^\alpha))$. Composing these maps we get a $G$-map
$\psi\circ\phi: \ell^\infty(G/\Lambda, I(\sigma(A^\alpha))) \to I(\sigma(A^\alpha))$. Since $G$ acts trivially on $B(H_\sigma)$, the restriction of the above map yields a $G$-map
$\theta: \ell^\infty(G/\Lambda) \to I(\sigma(A^\alpha))$. On the other hand, the restriction of the $G$-projection: $B(H)=B(\ell^2(G,H_\sigma))\to \mathcal{B}_{\tilde\sigma,\lambda}$ to $I(\sigma(A^\alpha))$ is a $G$-map $\beta: I(\sigma(A^\alpha))\to \mathcal{B}_{\tilde\sigma,\lambda}$, which gives us the $G$-map $\theta\circ\beta: \ell^\infty(G/\Lambda) \to \mathcal{B}_{\tilde\sigma,\lambda}$. Now by Proposition \ref{main1}, there is a $G$-projection: $\mathcal{B}_{\tilde\sigma,\lambda}\to C(\partial_FG)$. Composing  $\theta\circ\beta$ with this map gives a $G$-map $\gamma: \ell^\infty(G/\Lambda) \to C(\partial_FG)$. Finally, $G$ acts on $\partial_FG$ and each $G$-invariant Radon measure $\mu$ on $\partial_FG$ gives a $G$-invariant state $\varepsilon: C(\partial_FG)\to \mathbb C$. Finally, $\varepsilon\circ \gamma$ is a $G$-state on  $\ell^\infty(G/\Lambda)$, implying that $\Lambda$ is coamenable. 	
\end{proof}

\begin{corollary} \label{coa}
	A
	normal subgroup  $\Lambda\unlhd G$ is coamenable iff it is ${\rm Ind}(\sigma)$-coamenable, for some (any) representation $\sigma$ of $A$.
\end{corollary}

If $(\pi, u, H)$  be a covariant representation such that $u$ restricts to an irreducible representation on $\Lambda$, then $\Lambda$ is $(\pi, u)$-coamenable, since in this case, $u(\Lambda)^{'}=\mathbb C{\rm id}_H$, and  $I(\pi(A^\alpha))$ is always unital (even if $\pi(A^\alpha)$ is not unital).

  \begin{proposition} \label{amcoam}
  Let	$(\pi, u, H)$  be a covariant representations of $ (G,A,\alpha)$. If $u$ is amenable, then each
  	normal subgroup  $\Lambda\unlhd G$ is both $u$-amenable and $(\pi, u)$-coamenable. Conversely, if there is a normal subgroup  $\Lambda\unlhd G$ which is both $u$-amenable and $(\pi, u)$-coamenable and there is a $G$-invariant state on $I(\pi(A^\alpha))$, then $u$ is amenable.
  \end{proposition}
\begin{proof}
If $u$ is amenable, then any normal subgroup is $u$-amenable \cite[Proposition 4.2]{bk}. Also if $\omega$ is a $G$-invariant state on $B(H)$ and $\omega_0$ is its restriction to $u(\Lambda)^{'}$ is a $G$-invariant state on $u(\Lambda)^{'}$. Since $I(\pi(A^\alpha))$ is unital, a copy of $\mathbb C$ sits inside $I(\pi(A^\alpha))$, and so $\omega_0$ could be regarded as a $G$-map: $u(\Lambda)^{'} \to I(\pi(A^\alpha))$, that is, $\Lambda$ is  $(\pi, u)$-coamenable.

Conversely, suppose $\Lambda$ is a $u$-amenable normal subgroup which is also $(\pi, u)$-coamenable, then by Lemma \ref{t}, it acts trivially on $\mathcal{B}_{\pi,u}$, which means that
$\mathcal{B}_{\pi,u}\subseteq u(\Lambda)^{'}$ in $B(H)$. Now  the $G$-map: $u(\Lambda)^{'}\to I(\pi(A^\alpha))$ restricts to a $G$-map: $\mathcal{B}_{\pi,u}\to I(\pi(A^\alpha))$. Composing this with a $G$-idempotent: $B(H)\to \mathcal{B}_{\pi,u}$, we get a $G$-map: $B(H)\to I(\pi(A^\alpha))$. Compositing this with any $G$-invariant state on $I(\pi(A^\alpha))$ (which exists by assumption) we get a $G$-invariant state on $B(H)$, thus $u$ is amenable. 	
		
\end{proof}

In particular, if $u$ is a non amenable
unitary representation of $G$ whose restriction to a normal subgroup $\Lambda$ is reducible, then $\Lambda$ is not $u$-amenable (c.f., \cite[Corollary 4.13]{bk}).

\section{Extensions} \label{e}

The Furstenberg boundary has the following extension property in topological dynamics: if $\Lambda$ is a normal subgroup of $G$, then any action of $\Lambda$ on $\partial_FG$ extends to an action of $\Lambda$ on $\partial_FG$. In this section we seek analogous extension properties of the Furstenberg-Hamana boundary $\mathcal{B}_{\pi,u}$.

Let	$(\pi, u, H)$  be a covariant representations of $ (G,A,\alpha)$, where $A$ is unital. For an automorphism $\tau\in {\rm Aut}(G)$, we have the action $\alpha\circ \tau$ of $G$ on $A$ and $(\pi, u, H)$  is a covariant representations of $ (G,A,\alpha\circ\tau)$. An automorphism $\tau\in {\rm Aut}(G)$ is
called a $(\pi, u)$-automorphism if it extends to a C*-automorphism of $A\rtimes_{\pi,u} G$. We
denote the set of all such $G$-automorphisms by ${\rm Aut}_{\pi,u}(G)$. Note that $\tau\in {\rm Aut}_{\pi,u}(G)$ if and only if $(\pi,u, H)\sim (\pi^{'}, u\circ\tau, H)$, where $\pi^{'}$ is a modification of $\pi$ which makes a covariant pair with $u\circ\tau$ for the new action $\alpha\circ\tau$, and in this case, ${\rm Aut}_{\pi,u}(G)={\rm Aut}_{\pi^{'},u\circ\tau}(G)$.  This simply means that the automorphism on $A\rtimes_{\rm max} G$ induced by $\tau$ leaves the kernel of the
canonical *-epimorphism: $A\rtimes_{\rm max} G\to A\rtimes_{\pi, u} G$ invariant.

For the rest of this section, $A$ is assumed to be unital.

\begin{lemma} \label{ind}
For each representation $\sigma$ of $A$, ${\rm Aut}_{{\rm Ind}(\sigma)}(G)={\rm Aut}(G)$.
\end{lemma}
\begin{proof}
If $\tau$ is an automorphism of $G$, in the induced covariant representations $(\tilde\sigma, \lambda, \ell^2(G, H_\sigma))$ and $(\tilde\sigma, \lambda\circ\tau, \ell^2(G, H_\sigma))$, the representation $\tilde\sigma$ is calculated from $\sigma$  using actions $\alpha$ and $\alpha\circ\tau$, respectively. For a finite sum $a=\sum a_tt$ in $A[G]$, let $a_\tau:=\sum a_{\tau^{-1}(t)}t$, then
	\begin{align*}
		(\tilde\sigma\rtimes \lambda\circ\tau)(a)&=\sum \tilde\sigma(a_t)\lambda_{\tau(t)}=\sum \sigma(\alpha_{\tau(t^{-1})}(a_t)\lambda_{\tau(t)})
		\\&=\sum \sigma(\alpha_{t^{-1}}(a_{\tau^{-1}(t)})\lambda_{t}=(\tilde\sigma\rtimes \lambda)(a_\tau),
	\end{align*}
thus there is a C*-automorphism from $A\rtimes_{\tilde\sigma,\lambda\circ\tau} G$ onto $A\rtimes_{\tilde\sigma,\lambda} G$, sending $\lambda_{\tau(t)}$ to $\lambda_t$. 	
\end{proof}

Any inner automorphism is a $(\pi, u)$-automorphism for any covariant representation $(\pi, u, H)$. Thus, $G/Z(G)$ is identified with a normal subgroup of
${\rm Aut}_{\pi,u}(G)$. Since by Proposition \ref{trivial2}, $Z(G)$ acts trivially on  $\mathcal{B}_{\pi,u}$, it follows that the action of $G$ on $\mathcal{B}_{\pi,u}$ factors through $G/Z(G)$.

\begin{lemma} \label{ext}
	The action of $G$ on $\mathcal{B}_{\pi,u}$ extends
	to an action of ${\rm Aut}_{\pi,u}(G)$
	on $\mathcal{B}_{\pi,u}$.
\end{lemma}
\begin{proof}
	 We adapt the proof of \cite[Theorem 5.3]{bk}. Let	$(\pi, u, H)$  be a covariant representations of $ (G,A,\alpha)$ and let $\tau\in{\rm Aut}_{\pi,u}(G)$. Then $(\pi,u, H)\sim (\pi^{'}, u\circ\tau, H)$, where $\pi^{'}$ is an appropriate modification of $\pi$, making the latter a covariant representation of $ (G,A,\alpha\circ\tau)$. Take a *-isomorphism $\tilde\tau: A\rtimes_{\pi, u} G\to A\rtimes_{\pi^{'}, u\circ\tau} G$ sending $\pi(a)u_g$ to $\pi^{'}(a)u_{\tau(g)}$, for each $a\in A$ and $g\in G$. By Lemma \ref{wc}, there is a *-isomorphism $\hat\tau: \mathcal{B}_{\pi,u} \to \mathcal{B}_{\pi^{'},u\circ\tau}$ satisfying
	 $$\hat\tau(u_gxu_g^*)=u_{\tau(g)}\hat\tau(x)u_{\tau(g)}^*$$
for $g\in G$ and $x\in\mathcal{B}_{\pi,u}$. which shows that
$\hat\tau\circ\hat\tau^{-1}$ is a $G$-map on $\mathcal{B}_{\pi,u}$, and so is  identity by $G$-rigidity. Similarly, $\hat{\tau}_2^{-1}\circ\hat{\tau}_1\circ\hat{\tau}_1\hat{\tau}_2$ is identity on $\mathcal{B}_{\pi,u}$, for $\tau_1, \tau_2\in{\rm Aut}_{\pi,u}(G)$.  Thus ${\rm Aut}_{\pi,u}(G)$ acts	on $\mathcal{B}_{\pi,u}$. Let $\tau_g=$Ad$_g$ be the inner automorphism at $g\in G$, then a simple calculation shows that $\hat\tau_g\circ$Ad$_{u_{g}^{*}}$ is a $G$-map on $\mathcal{B}_{\pi,u}$, which again has to be identity, hence $\hat\tau_g=$Ad$_{u_{g}}$, and we are done. 	
\end{proof}

Observe that the above argument also shows that the extension is unique as long as it satisfies the relation
$$\hat\tau(u_gxu_g^*)=u_{\tau(g)}\hat\tau(x)u_{\tau(g)}^*$$
for $g\in G$ and $x\in\mathcal{B}_{\pi,u}$. From now on we talk about such an extension.

Recall that Rad$_{\pi,u}(G)$ is kernel of the action of $G$ on $\mathcal{B}_{\pi,u}$. then by the above relation, each $\tau\in{\rm Aut}_{\pi,u}(G)$ keeps Rad$_{\pi,u}(G)$ invariant. This induces a group homomorphism
$$\kappa: {\rm Aut}_{\pi,u}(G)\to {\rm Aut}_{\pi,u}\big(G/{\rm Rad}_{\pi,u}(G)\big),$$
whose kernel is nothing but the kernel of the above unique extended action (again by the above relation). In particular, if the action of $G$ on $\mathcal{B}_{\pi,u}$ is faithful, then so is its unique extension to an action of ${\rm Aut}_{\pi,u}(G)$
on $\mathcal{B}_{\pi,u}$, satisfying the above relation.

\begin{proposition} \label{subext}
Let	$(\pi, u, H)$  be a covariant representations of $ (G,A,\alpha)$ and $\Lambda$ be a normal subgroup of $G$, and let $u_0$ be the restriction of $u$ on $\Lambda$. Then the action of $\Lambda$
on $\mathcal{B}_{\pi,u_0}$ extends to an action of $G$
on $\mathcal{B}_{\pi,u_0}$. When the $\Lambda$-action is faithful, the kernel of the $G$-action is the inverse image of the centralizer of $u(\Lambda)$ in $u(G)$.
\end{proposition}
\begin{proof}
Since $\Lambda$ is normal, the C*-algebra $A\rtimes_{\pi,u_0} G$ is $G$-invariant and Ad$_{u_{g}}$ is in ${\rm Aut}_{\pi,u_0}(G)$, for $g\in G$.
This implements a group homomorphism $\vartheta: G\to  {\rm Aut}_{\pi,u_0}(G)$ whose kernel is the inverse image of the centralizer of $u(\Lambda)$ in $u(G)$. When we have the
faithfulness of the $\Lambda$-action, the action of $G/ker(\vartheta)$ on $\mathcal{B}_{\pi,u_0}$ is faithful and so the the kernel of the $G$-action is nothing but $ker(\vartheta)$.
	
\end{proof}

\section{Traces on crossed products}\label{tr}

The groups with the unique trace property are studied in  \cite{bkko}.
 Every trace on the reduced C*-algebra $C_r^*(G)$ of a discrete group $G$  is known to be supported on the amenable radical
Rad$(G)$ \cite[Theorem 4.1]{bkko}, which is the kernel of the action of $G$ on the Furstenberg boundary $\partial_FG$.

In this section we study traces on reduced crossed products $A\rtimes_{\rm red} G$. In fact, we could do this for the slightly more general crossed products $A\rtimes_{{\rm Ind}(\sigma)} G$, for any representation $(\sigma, H_\sigma)$ of $A$. A trace  on $A\rtimes_{{\rm Ind}(\sigma)} G$ is nothing but a $G$-map $\tau: A\rtimes_{{\rm Ind}(\sigma)}G\to\mathbb C$, where
$\mathbb C$ carries the trivial action of $G$. By $G$-injectivity, $\tau$ extends to a $G$-map $\tilde\tau: B\big(\ell^2(G, H_\sigma)\big)\to C(\partial_FG)$. By Proposition \ref{main1}, we may regard $\tilde\tau$ as a $G$-map $\tilde\tau: B\big(\ell^2(G, H_\sigma)\big)\to \mathcal{B}_{\tilde\sigma,\lambda}$. Now a trace $\tau$ on $A\rtimes_{{\rm Ind}(\sigma)} G$ is  supported on the normal subgroup Rad$_{\tilde\sigma,\lambda}(G)\unlhd G$ iff the extension
$\tilde\tau: B\big(\ell^2(G, H_\sigma)\big)\to \mathcal{B}_{\tilde\sigma,\lambda}$ vanishes outside the  kernel of the action of $G$ on $\mathcal{B}_{\tilde\sigma,\lambda}$.

Let	$(\pi, u, H)$  be a covariant representations of $ (G,A,\alpha)$. Let us call a trace on $A\rtimes_{\pi,u} G$  {\it extendable} if it extends to a $G$-map:  $B(H)\to  \mathcal{B}_{\pi,u}$. The above observation means that any trace on $A\rtimes_{{\rm Ind}(\sigma)} G$ is extendable.

The next lemma extends \cite[Lemma 2]{as}, as well as \cite[Lemma 2.2]{hk}.

\begin{lemma} \label{rad2}
	Let	$(\pi, u, H)$  be a covariant representation,
	$\psi: B(H) \to \mathcal{B}_{\pi,u}$ be a $G$-map and $g\notin$Rad$_{\pi,u}(G)$, then $\psi(u_g)$ is in the kernel of  some state on $\mathcal{B}_{\pi,u}$.
\end{lemma}
\begin{proof}
	Let $g\notin$Rad${}_{\pi,u}(G)$, and choose $x\in \mathcal{B}_{\pi,u}$ with $x\neq u_gxu_g^*$.
	We may assume that $0\leq x\leq 1$. Choose a state $\omega$ on $\mathcal{B}_{\pi,u}$ with $\omega(x)=1$ and $\omega(u_gxu_g^*)=0$.
	Since $\psi$ acts as identity on $\mathcal{B}_{\pi,u}$ by Proposition \ref{b}$(iv)$, $\omega\circ \psi(x)=1$ and $ \omega\circ\psi(u_gxu_g^*)=0$. For the state $\rho=\omega\circ \psi$, we  have $ \rho(u_gx^2u_g^*)= \rho(1-x^2)=0$,  and so by Cauchy-Schwartz inequality,  $ \rho\big((1-x)u_g\big)= \rho\big(xu_g\big)=0$, thus
	$$ \rho(u_g)=\rho\big((1-x)u_g\big)+ \rho\big(xu_g\big)=0,$$
	that is, $\psi(u_g)$ is in the kernel of $\omega$.
		
\end{proof}

A trace  on $A\rtimes_{\pi,u} G$ is nothing but a $G$-map $\tau: A\rtimes_{\pi,u}G\to\mathbb C$, where
$\mathbb C$ carries the trivial action of $G$ by Ad$_u$. We say that $\tau$ is supported on a normal subgroup $\Lambda\unlhd G$ if $\tau(u_g)=0$, for $g\notin \Lambda$.

\begin{proposition} \label{rad3}
	Let	$(\pi, u, H)$  be a covariant representations of $ (G,A,\alpha)$. Then any extendable trace on $A\rtimes_{\pi,u} G$ is supported on Rad$_{\pi,u}(G)$.
\end{proposition}
\begin{proof}
Let $\tau$ be an extendable trace on $A\rtimes_{\pi,u} G$ and consider the $G$-extension $\tilde\tau: B(H) \to \mathcal{B}_{\pi,u}$. Then, $\tilde\tau(u_g)=\tau(u_g)1$, for any $g\in G$. If $g\notin$Rad$_{\pi,u}(G)$, then by Lemma \ref{rad2}, $\tilde\tau(u_g)$ is in the kernel some state $\omega$ on $\mathcal{B}_{\pi,u}$. Hence, $\tau(u_g) = \omega(\tau(u_g)1)=\omega(\tilde\tau(u_g))=0$.
\end{proof}

Recall that when $A$ is unital, for a covariant representation $(\pi,u,H)$, $A\rtimes_{\pi,u} G$ is nothing but the C*-algebra generated by $\pi(A)\cup u(G)$ in $B(H)$. In this case, there is a copy of $C^*_u(G)$ inside  $A\rtimes_{\pi,u} G$. In general, there is a copy of $C^*_u(G)$ inside  $M(A\rtimes_{\pi,u} G)$.

\begin{corollary} \label{rad4}
	Let	$(\pi, u, H)$  be a covariant representation of $ (G,A,\alpha)$. If $G$ acts faithfully on $\mathcal{B}_{\pi,u}$, then $A\rtimes_{\pi,u} G$  either has no  extendable trace, or each two extendable traces on $A\rtimes_{\pi,u} G$ agree on $C^*_u(G)$.
\end{corollary}
\begin{proof}
	If we have Rad$_{\pi,u}(G)=\{e\}$, and if there is any extendable trace $\tau$, then $\tau(u_g)=0$, for $g\neq e$. Since a trace on $A\rtimes_{\pi,u} G$ extend to a functional on $M(A\rtimes_{\pi,u} G)$, this means that any two extendable traces  agree on the copy of $C^*_u(G)$ inside  $M(A\rtimes_{\pi,u} G)$.  
\end{proof}

Back to the second example discussed after Definition \ref{b}, if $G$ has trivial amenable radical and the centralizer of each element is amenable (again the free group $\mathbb F_2$ has both properties) then the quasi-regular representations
associated to stabilizer subgroups are weakly contained in the regular representation and the restriction of the representation $u_g:=\lambda_g^G\rho_g^G$ to $\ell^2(G\backslash\{e\})$ is unitarily equivalent to the direct sum of these quasi-regular representations, and so is weakly contained in the regular representation, and so is $u$. Now $G$ acts faithfully on $\partial_FG$ (as the amenable radical is trivial). If moreover it act faithfully on $\mathcal{B}_{\iota_0,\lambda^G}$, then $C_{u_0}^*(G)\rtimes_{\iota_0,\lambda^G}G$ either has no extendable trace or any two extendable traces agree on $C_{\lambda}^*(G)$, by Corollary \ref{rad4}. Indeed, in this example, since the action of $G$ on $C_u^*(G)$ is inner (and so exterior equivalent to the trivial action), the crossed product $C_{u_0}^*(G)\rtimes_{\iota_0,\lambda^G}G$ is isomorphic with $C_{u_0}^*(G)\otimes_{\gamma} C_{\lambda}^*(G)$  (c.f., \cite[Remark 3.9]{pi}), for the C*-norm $\gamma$ defined by
$$\|a\otimes b\|_\gamma:=\|(\iota_0\rtimes\lambda_G)(a\otimes b)\|\ \ (a,b\in C_c(G)),$$
and extendable traces are tensors of traces on each factor.

\vspace{.2cm}
{\bf Data availability statement:}
Data sharing not applicable to this article as no datasets were generated or analysed during the current study.

-----------------------------------------------------------------------------------


\end{document}